\documentclass{amsart}
\usepackage{amssymb, amsmath, amsthm, graphics, comment, xspace, enumerate}
\newtheorem{thm}{Theorem}[section]

\newtheorem{lem}[thm]{Lemma}
\newtheorem{prop}[thm]{Proposition}
\theoremstyle{definition}
\newtheorem{defn}[thm]{Definition}

\theoremstyle{remark}
\newtheorem{rem}[thm]{Remark}
\numberwithin{equation}{section}

\begin{document}
\title[Composition principles for generalized almost periodic...]{Composition principles for generalized almost periodic functions}

\author{Marko Kosti\' c}
\address{Faculty of Technical Sciences,
University of Novi Sad,
Trg D. Obradovi\' ca 6, 21125 Novi Sad, Serbia}
\email{marco.s@verat.net}

{\renewcommand{\thefootnote}{} \footnote{
\\ \text{  }  \ \    {\it Key words and phrases.} Stepanov $p$-almost periodic type functions, Weyl $p$-almost periodic type functions, composition principles, abstract semilinear Cauchy inclusions, Banach spaces.
\\  \text{  }  \ \ The author is partially supported by grant 174024 of Ministry
of Science and Technological Development, Republic of Serbia.}}

\begin{abstract}
In this paper, we consider composition principles for generalized almost periodic functions. We prove several new composition principles for the classes of (asymptotically) Stepanov $p$-almost periodic functions and (asymptotically, equi-)Weyl $p$-almost periodic functions, where $1\leq p<\infty ,$ and explain how we can use some of them in the qualitative analysis of solutions for certain classes of abstract semilinear Cauchy inclusions in Banach spaces. 
\end{abstract}

\maketitle

\section[Introduction and preliminaries]{Introduction and preliminaries}\label{intro}

The theory of generalized almost periodic and generalized almost automorphic functions is still an active field of research. Composition principles for generalized almost periodic and generalized almost automorphic functions, sometimes also called superposition principles, play an important role in the qualitative theory of abstract semilinear integro-differential equations in Banach spaces (see \cite{diagana}-\cite{gaston}, \cite{nova-mono} and references cited therein for more details on the subject).
The main aim of this paper is to consider composition principles for generalized almost periodic functions, only. As mentioned in the abstract, we focus our attention to the  Stepanov $p$-almost periodic type functions and Weyl $p$-almost periodic type functions, where $1\leq p<\infty .$

The organization and main ideas of paper are briefly described as follows. In Subsection \ref{subs}, we recall the basic definitions and results from the theory of generalized almost periodic functions. Our first contribution is Theorem \ref{vcb-show}, in which we reconsider a recent result of F. Bedouhene, Y. Ibaouene, O. Mellah and P. Raynaud de Fitte \cite[Theorem 3]{mellah} for the class of Stepanov $p$-almost periodic functions (see also Theorem \ref{vcb-show-weyll}, where we clarify a related result for the class of generalized Weyl $p$-almost periodic functions in the sense of Kovanko's approach \cite{kovanko}). We continue by stating and proving Proposition \ref{biblonja}, in which we examine an analogue of Theorem \ref{vcb-show} for asymptotically Stepanov $p$-almost periodic functions. The main result of paper is Theorem \ref{vcb-prim}, in which we clarify a new composition principle for the class of (equi-)Weyl $p$-almost periodic functions. The proof of Theorem \ref{vcb-prim}, which is much simpler than that of \cite[Theorem 3]{mellah} and works also for both classes of Weyl $p$-almost periodic functions under our consideration, is based on the method proposed by W. Long and H.-S. Ding in the proof of \cite[Theorem 2.2]{comp-adv} for Stepanov class. 
The main purpose of Theorem \ref{rora} is to analyze a composition principle for asymptotically (equi-)Weyl $p$-almost periodic functions. Finally, Subsection \ref{semilinear} is deserved for applications of our abstract theoretical results to abstract semilinear Cauchy inclusions in Banach spaces.

We use the standard terminology throughout the paper. By $X$ and $Y$ we denote two non-trivial complex Banach spaces; by
$\|x\|,$ $\|y\|_{Y}$ and $\|T\|_{L(X,Y)}$ we denote the norm of an element $x\in X,$ an element $y\in Y$ and a continuous linear mapping $T\in L(X,Y),$ respectively. Unless specified otherwise, we assume that $1\leq p<\infty$ and $I={\mathbb R}$ or $I=[0,\infty)$ henceforth.
By $C_{b}(I : X)$ and $C_{0}(I : X)$ we denote the vector spaces consisting of all bounded continuous functions $f  : I \rightarrow X$ and all bounded continuous functions $f  : I \rightarrow X$ such that $\lim_{|t|\rightarrow +\infty}\|f(t)\|=0.$ We refer the reader to \cite{nova-mono} for the notions of Weyl Liouville fractional derivatives and Caputo fractional derivatives used in Subsection \ref{semilinear}.

\subsection{Generalized almost periodic type functions}\label{subs}

Let $f : I \rightarrow X$ be continuous. Given a real number $\epsilon>0$ in advance, we call $\tau>0$ an $\epsilon$-period for $f(\cdot)$ iff
$
\| f(t+\tau)-f(t) \| \leq \epsilon, $ $ t\in I.
$
By $\vartheta(f,\epsilon)$ we denote the set consisting of all $\epsilon$-periods for $f(\cdot).$ We say that $f(\cdot)$ is almost periodic iff for each $\epsilon>0$ the set $\vartheta(f,\epsilon)$ is relatively dense in $I,$ which means that
there exists $l>0$ such that any subinterval of $I$ of length $l$ meets $\vartheta(f,\epsilon)$. The space consisting of all almost periodic functions from the interval $I$ into $X$ will be denoted by $AP(I:X).$

The class of asymptotically almost periodic functions was introduced by M. Fr\' echet in 1941, for the case that $I=[0,\infty)$ (more details about this class of functions with values in Banach spaces can be found in \cite{cheban}-\cite{diagana}, \cite{gaston} and references cited therein). A function $f \in C_{b}([0,\infty) : X)$ is said to asymptotically almost periodic iff
for every $\epsilon >0$ we can find numbers $ l > 0$ and $M >0$ such that every subinterval of $[0,\infty)$ of
length $l$ contains, at least, one number $\tau$ such that $\|f(t+\tau)-f(t)\| \leq \epsilon$ provided $t \geq M.$
The space consisting of all asymptotically almost periodic functions from $[0,\infty)$ into $X$ is denoted by
$AAP([0,\infty) : X).$ It is well known that, for a function $f \in C_{b}([0,\infty):X),$ the following
statements are equivalent: 
\begin{itemize}
\item[(i)] $f\in AAP([0,\infty) :X).$
\item[(ii)] There exist uniquely determined functions $g \in AP([0,\infty) :X)$ and $\phi \in  C_{0}([0,\infty): X)$
such that $f = g+\phi.$
\end{itemize}
In this paper, we will not consider asymptotically almost periodic functions defined on the real line.

Let $l>0$ and $f,\ g\in L^{p}_{loc}(I :X).$ We define the Stepanov `metric' by
\begin{align*}
D_{S_{l}}^{p}\bigl[f(\cdot),g(\cdot)\bigr]:= \sup_{x\in I}\Biggl[ \frac{1}{l}\int_{x}^{x+l}\bigl \| f(t) -g(t)\bigr\|^{p}\, dt\Biggr]^{1/p}.
\end{align*}
The Stepanov and Weyl `norm' of $f(\cdot)$ are defined by\index{Stepanov norm} \index{Stepanov distance} \index{Weyl!distance} \index{Weyl!norm}
$$
\bigl\| f  \bigr\|_{S_{l}^{p}}:= D_{S_{l}}^{p}\bigl[f(\cdot),0\bigr]\ \mbox{  and  }\ \bigl\| f  \bigr\|_{W^{p}}:= D_{W}^{p}\bigl[f(\cdot),0\bigr]:=\lim_{l\rightarrow +\infty}D_{S_{l}}^{p}\bigl[f(\cdot),0\bigr],
$$
respectively. As it is well known, the last limit always exists in $[0,\infty].$

It is said that a function $f\in L^{p}_{loc}(I :X)$ is Stepanov $p$-bounded, $S^{p}$-bounded shortly, iff
$
\|f\|_{S^{p}}:=\sup_{t\in I}( \int^{t+1}_{t}\|f(s)\|^{p}\, ds)^{1/p}<\infty.
$
Let us recall that the function $f\in L^{p}_{loc}(I :X)$ is Stepanov $p$-bounded iff $f\in L^{p}_{loc}(I :X)$ is Weyl $p$-bounded, $W^{p}$-bounded shortly, i.e., if $
\|f\|_{W^{p}}<\infty.$ 
The space $L_{S}^{p}(I:X)$ consisted of all $S^{p}$-bounded functions becomes a Banach space equipped with the above norm.
We say that a function $f\in L_{S}^{p}(I:X)$ is (asymptotically) Stepanov $p$-almost periodic iff the function
$
\hat{f} : I \rightarrow L^{p}([0,1] :X),
$ defined by
$
\hat{f}(t)(s):=f(t+s),$ $ t\in I,\ s\in [0,1]
$
is (asymptotically) almost periodic.
By $APS^{p} (I: X)$ and $AAPS^{p} (I: X)$ we denote the spaces consisted of all $S^{p}$-almost periodic functions $I\mapsto X$ and asymptotically $S^{p}$-almost periodic functions $I\mapsto X,$ respectively. In the case that $p=\infty,$ we set $L_{S}^{\infty}(I : X)\equiv L^{\infty}(I: X) ;$ $L_{S}^{r}(I)\equiv L^{r}(I: {\mathbb C} )$ for $r\in [1,\infty].$ 

The notion of an (equi-)Weyl $p$-almost periodic function is given as follows (see \cite{deda}, \cite{nova-mono} and references cited therein for more details on the subject):

\begin{defn}\label{Weyl defn}
Let $f\in L_{loc}^{p}(I: X).$ \index{function!equi-Weyl $p$-almost periodic} \index{function!Weyl $p$-almost periodic}
\begin{itemize}
\item[(i)] We say that the function $f(\cdot)$ is equi-Weyl $p$-almost periodic, $f\in e-W_{ap}^{p}(I:X)$ for short, iff for each $\epsilon>0$ we can find two real numbers $l>0$ and $L>0$ such that any interval $I'\subseteq I$ of length $L$ contains a point $\tau \in  I'$ such that
\begin{align*}
\sup_{x\in I}\Biggl[ \frac{1}{l}\int_{x}^{x+l}\bigl \| f(t+\tau) -f(t)\bigr\|^{p}\, dt\Biggr]^{1/p} \leq \epsilon, \mbox{ i.e., } D_{S_{l}}^{p}\bigl[f(\cdot+\tau),f(\cdot)\bigr] \leq \epsilon.
\end{align*}
\item[(ii)] We say that the function $f(\cdot)$ is Weyl $p$-almost periodic, $f\in W_{ap}^{p}(I: X)$ for short, iff for each $\epsilon>0$ we can find a real number $L>0$ such that any interval $I'\subseteq I$ of length $L$ contains a point $\tau \in  I'$ such that
\begin{align*}
 \sup_{x\in I}\Biggl[ \frac{1}{l}\int_{x}^{x+l}\bigl \| f(t+\tau) -f(t)\bigr\|^{p}\, dt\Biggr]^{1/p} \leq \epsilon, \mbox{ i.e., } \lim_{l\rightarrow \infty}D_{S_{l}}^{p}\bigl[f(\cdot+\tau),f(\cdot)\bigr] \leq \epsilon.
\end{align*}
\end{itemize}
\end{defn}

Asymptotically (equi-)Weyl $p$-almost periodic functions have been recently introduced by the author in \cite{weyl}.
If $q\in L_{loc}^{p}([0,\infty ) : X),$ then we define the function ${\bf q}(\cdot,\cdot): [0,\infty) \times [0,\infty) \rightarrow X$ by
$
{\bf q}(t,s):=q(t+s),$ $t,\, s\geq 0.
$ 

\begin{defn}\label{stea-weyl}
It is said that $q\in L_{loc}^{p}([0,\infty ) : X)$ is Weyl $p$-vanishing iff
\begin{align}\label{radnasebi}
\lim_{t\rightarrow \infty}\bigl\|{\bf q}(t,\cdot)\bigr\|_{W^{p}}=0,\mbox{ i.e., }\lim_{t\rightarrow \infty}\, \lim_{l\rightarrow \infty}\ \ \sup_{x\geq 0}\Biggl[ \frac{1}{l}\int_{x}^{x+l}\bigl \| q(t+s)\bigr\|^{p}\, ds\Biggr]^{1/p}=0.
\end{align}
\end{defn}

Replacing the limits in \eqref{radnasebi}, we 
come to the class of equi-Weyl $p$-vanishing functions. 
It is said that a function $q\in L_{loc}^{p}([0,\infty ) : X)$ is equi-Weyl $p$-vanishing iff
\begin{align*}
\lim_{l\rightarrow \infty}\, \lim_{t\rightarrow \infty}\ \ \sup_{x\geq 0}\Biggl[ \frac{1}{l}\int_{x}^{x+l}\bigl \| q(t+s)\bigr\|^{p}\, ds\Biggr]^{1/p}=0.
\end{align*}

We will use the following elementary lemma (see e.g. \cite[p. 70]{besik} for the scalar-valued case):

\begin{lem}\label{bes}
Let $-\infty<a<b<\infty,$ let $1\leq p'<p''<\infty,$ and let $f\in L^{p''}([a,b]: X).$ Then $f\in L^{p'}([a,b]: X)$ and
$$
\Biggl[\frac{1}{b-a}\int^{b}_{a}\|f(s)\|^{p'}\, ds \Biggr]^{1/p'}\leq \Biggl[\frac{1}{b-a}\int^{b}_{a}\|f(s)\|^{p''}\, ds \Biggr]^{1/p''}.
$$
\end{lem}

In the remaining part of this section, we will analyze generalized (asymptotically) almost periodic type functions depending on two parameters.
Denote by $C_{0}([0,\infty) \times Y : X)$ the space consisting of all
continuous functions $h : [0,\infty)  \times Y \rightarrow X$ such that $\lim_{t\rightarrow \infty}h(t, y) = 0$ uniformly for $y$ in any compact subset of $Y .$ 
 If $f : I  \times Y  \rightarrow X,$ then we define $\hat{f} : I  \times Y  \rightarrow L^{p}([0,1]:X)$ by $\hat{f}(t , y):=f(t +\cdot, y),$ $t\geq 0,$ $y\in Y.$

We recall the following definition (cf. \cite{nova-mono} and references cited therein):

\begin{defn}\label{definicija}
Let $1\leq p <\infty.$ 
\begin{itemize}
\item[(i)]
A function $f : I \times Y \rightarrow X$ is said to be almost periodic iff $f (\cdot, \cdot)$ is bounded, continuous as well as for every $\epsilon>0$ and every compact
$K\subseteq Y$ there exists $l(\epsilon,K) > 0$ such that every subinterval $J\subseteq I$ of length $l(\epsilon,K)$ contains a number $\tau$ with the property that $\|f (t +\tau , y)- f (t, y)\| \leq \epsilon$ for all $t \in  I,$ $ y \in K.$ The collection of such functions will be denoted by $AP(I \times Y : X).$
\item[(ii)] A function $f : [0,\infty)  \times Y \rightarrow X$ is said to be asymptotically almost periodic iff it is bounded, continuous and admits a
decomposition $f = g + q,$ where $g \in AP([0,\infty)  \times Y : X)$ and $q\in C_{0}([0,\infty)  \times Y : X).$ Denote by
 $AAP([0,\infty)  \times Y : X) $ the vector space consisting of all such functions.
\item[(iii)] A function $f : I \times Y \rightarrow X$ is said to be Stepanov $p$-almost periodic, $S^{p}$-almost periodic shortly, iff $\hat{f} : I  \times Y  \rightarrow L^{p}([0,1]:X)$ is almost periodic. Denote by
 $APS^{p}(I  \times Y : X) $ the vector space consisting of all such functions.
\item[(iv)] A function $f : [0,\infty)  \times Y \rightarrow X$
is said to be asymptotically $S^p$-almost periodic
iff $\hat{f}: [0,\infty)  \times Y \rightarrow  L^{p}([0,1]:X)$ is asymptotically almost periodic. The collection of such functions will be denoted by $AAPS^{p}([0,\infty) \times Y : X).$
\end{itemize}
\end{defn}

The following definition is slightly different from the corresponding definitions introduced recently in \cite{mellah} and \cite{irkutsk-prim} for the class of equi-Weyl $p$-almost periodic functions, with only one pivot space $X=Y:$

\begin{defn}\label{definicija}
\begin{itemize}
\item[(i)] 
A function $f : I \times Y \rightarrow X$ is said to be equi-Weyl $p$-almost periodic in $t\in I$ uniformly with respect to compact subsets of $Y$ iff
$f(\cdot,u)\in L_{loc}^{p}(I : X)$ for each fixed element $u\in Y$ and
if for each $\epsilon>0$ and each compact $K$ of $Y$ there exist two numbers $l>0$ and $L>0$ such that any interval $I'\subseteq I$ of length $L$ contains a point $\tau \in  I'$ such that
$$
\sup_{u\in K}\sup_{x \in I} \Biggl[\frac{1}{l}\int^{x+l}_{x}\bigl\|f(t+\tau,u)-f(t,u)\bigr\|^{p}\, dt\Biggr]^{1/p} <\epsilon . 
$$
We denote by $e-W_{ap,{\bf K}}^{p}(I \times Y:X)$ the vector space consisting of all such functions.
\item[(ii)] A function $f : I \times Y \rightarrow X$ is said to be Weyl $p$-almost periodic in $t\in I$ uniformly with respect to compact subsets of $Y$ iff
$f(\cdot,u)\in L_{loc}^{p}(I : X)$ for each fixed element $u\in Y$ and
if for each $\epsilon>0$ and each compact $K$ of $Y$
we can find a real number $L>0$ such that any interval $I'\subseteq I$ of length $L$ contains a point $\tau \in  I'$ satisfying that there exists a finite number $l(\epsilon,\tau)>0$ such that
\begin{align*}
\sup_{u\in K}\sup_{x\in I}\Biggl[ \frac{1}{l}\int_{x}^{x+l}\bigl \| f(t+\tau,u) -f(t,u)\bigr\|^{p}\, dt\Biggr]^{1/p} < \epsilon,\quad l\geq l(\epsilon,\tau).
\end{align*}
We denote by $W_{ap,{\bf K}}^{p}(I \times Y:X)$ the vector space consisting of all such functions.
\end{itemize}
\end{defn}

The following definition is known in the case that $X=Y$ (cf. \cite{irkutsk-prim}):

\begin{defn}\label{stea}
Let $q : [0,\infty ) \times Y \rightarrow X$ be such that $q(\cdot,u)\in L_{loc}^{p}([0,\infty) : X)$ for each fixed element $u\in Y.$
\begin{itemize}
\item[(i)] It is said that $q(\cdot,\cdot)$
is Weyl $p$-vanishing uniformly with respect to compact subsets of $Y$ iff 
for each compact set $K$ of $Y$ we have:
\begin{align*}
\lim_{t\rightarrow \infty}\, \lim_{l\rightarrow \infty}\ \ \sup_{\xi\geq 0 , u\in K}\Biggl[ \frac{1}{l}\int_{\xi}^{\xi+l}\bigl \| q(t+s,u)\bigr\|^{p}\, ds\Biggr]
^{1/p}=0.
\end{align*}
\item[(ii)] It is said  that $q(\cdot,\cdot)$
is equi-Weyl $p$-vanishing uniformly with respect to compact subsets of $Y$ iff 
for each compact set $K$ of $Y$ we have:
\begin{align*}
\lim_{l\rightarrow \infty}\, \lim_{t\rightarrow \infty}\ \ \sup_{\xi\geq 0 , u\in K}\Biggl[ \frac{1}{l}\int_{\xi}^{\xi+l}\bigl \| q(t+s,u)\bigr\|^{p}\, ds\Biggr]
^{1/p}=0.
\end{align*}
\end{itemize}
We denote by $W^{p}_{0,{\bf K}}(I \times Y:X)$ and $e-W^{p}_{0,{\bf K}}(I \times Y:X)$ the classes consisting of all Weyl $p$-vanishing functions, uniformly with respect to compact subsets of $Y$ and all equi-Weyl $p$-vanishing functions, uniformly with respect to compact subsets of $Y$, respectively.
\end{defn}

\section[Formulation and proof of main results]{Formulation and proof of main results}\label{goti}

We start by stating Theorem \ref{vcb-show}, which has been recently considered by F. Bedouhene, Y. Ibaouene, O. Mellah and P. Raynaud de Fitte \cite[Theorem 3]{mellah} for the class of equi-Weyl $p$-almost periodic functions. To the best knowledge of the author, this result has not been considered elsewhere for the class of (asymptotically) Stepanov $p$-almost periodic functions. It should be noted that a well-known result of W. Long and H.-S. Ding \cite[Theorem 2.2]{comp-adv} is not comparable with Theorem \ref{vcb-show}. Simply said, with the notation used below, the assumption $p=q$ always imposes the value of exponent $r=\infty$ in \cite[Theorem 2.2]{comp-adv}, which is not a general case in examination; on the other hand, the assumptions of \cite[Theorem 2.2]{comp-adv} requires the value $p>1$ of the exponent, which is not generally needed for applying Theorem \ref{vcb-show}. Only we can say is that \cite[Theorem 3]{mellah} in the Stepanov setting, i.e., Theorem \ref{vcb-show}, is much more general than \cite[Theorem 2.7.2]{nova-mono}, where the usual Lipschitz condition for the value $p=1$ of the exponent has been analyzed. The proof of theorem is very similar to that of \cite[Theorem 2.2]{comp-adv} and we will include only the most relevant details.

\begin{thm}\label{vcb-show}
Suppose that $p,\ q\in [1,\infty) ,$ $r\in [1,\infty],$ $1/p=1/q+1/r$ and the following conditions hold:
\begin{itemize}
\item[(i)] $f \in APS^{p}(I \times Y : X)  $ and there exists a function $ L_{f}\in L_{S}^{r}(I) $ such that 
\begin{align}\label{vbnmp}
\| f(t,x)-f(t,y)\| \leq L_{f}(t)\|x-y\|_{Y},\quad t\in I,\ x,\ y\in Y.
\end{align}
\item[(ii)] $x \in  APS^{q} (I: Y),$ and there exists a set ${\mathrm E} \subseteq I$ with $m ({\mathrm E})= 0$ such that
$ K :=\{x(t) : t \in I \setminus {\mathrm E}\}$
is relatively compact in $X;$ here, $m(\cdot)$ denotes the Lebesgue measure.
\end{itemize}
Then $f(\cdot, x(\cdot)) \in  APS^{p}(I : X).$
\end{thm}

\begin{proof}
The measurability and $S^{p}$-boundedness of function $f(\cdot, x(\cdot))$ can be shown as in the proof of \cite[Theorem 2.2]{comp-adv}, by interchanging the role of exponents $p$ and $q$ in the computation \cite[p. 6, l.-1/l. -4]{comp-adv}. Using the same change of roles for the exponents $p$ and $q$ in the computation ending the proof of above-mentioned theorem, we get that
\begin{align*}
\Biggl( & \int^{1}_{0}\bigl\|f(t+s+\tau,x(t+s+\tau))-f(t+s,x(t+s))\bigr \|^{p} \, ds \Biggr)^{1/p}
\\ \leq & \bigl\|L_{f}\bigr\|_{L_{S}^{r}(I)}\bigl\|x(t+\tau+\cdot)-x(t+\cdot) \bigr\|_{L^{q}([0,1] : Y)}
\\ + & \Biggl( \int^{1}_{0}\Bigl(\sup_{u\in K}\|f(t+s+\tau,u)-f(t+s,u)\|\Bigr)^{p}\, ds \Biggr)^{1/p}
\end{align*}
for all $t\in I,$ with the meaning clear. Since $r\leq p,$ the argumentation given in the proof of \cite[Lemma 2.1]{comp-adv} shows that there is an absolute constant $M>0$ such that
$$
\Biggl( \int^{1}_{0}\bigl\|f(t+s+\tau,x(t+s+\tau))-f(t+s,x(t+s))\bigr\|^{p} \, ds \Biggr)^{1/p}\leq M(1+\bigl\|L_{f}\bigr\|_{L_{S}^{r}(I)})\epsilon ,
$$ 
finishing the proof.
\end{proof}

Now we are in a position to clarify the following composition principle for asymptotically Stepanov $p$-almost periodic functions (cf. also \cite[Proposition 2.7.3]{nova-mono} for a similar result in this direction):

\begin{prop}\label{biblonja}
Let $I =[0,\infty).$
Suppose that $p,\ q\in [1,\infty) ,$ $r\in [1,\infty],$ $1/p=1/q+1/r$ and the following conditions hold: 
\begin{itemize}
\item[(i)] $g \in APS^{p}(I \times Y : X)  ,$ $ L_{g}\in L_{S}^{r}(I) $ and \eqref{vbnmp} holds with the functions $f(\cdot,\cdot)$ and $L_{f}(\cdot)$ replaced with the functions $g(\cdot,\cdot)$ and $L_{g}(\cdot)$
therein.
\item[(ii)] $y \in APS^{q}(I:Y)$ and there exists a set $E \subseteq I$ with $m (E)= 0$ such that
$ K =\{y(t) : t \in I \setminus E\}$
is relatively compact in $Y.$
\item[(iii)] $f(t,u)=g(t,u)+q(t,u)$ for all $t\geq 0$ and $u\in Y,$ where $\hat{q}\in C_{0}([0,\infty) \times Y : L^{p}([0,1]:X)).$
\item[(iv)] $x(t)=y(t)+z(t) $ for all $t\geq 0,$ where $\hat{z}\in C_{0}([0,\infty) : L^{q}([0,1]:Y)).$
\item[(v)]  There exists a set $E' \subseteq I$ with $m (E')= 0$ such that
$ K' =\{x(t) : t \in I \setminus E'\}$
is relatively compact in $ Y.$
\end{itemize}
Then $f(\cdot, x(\cdot)) \in AAPS^{p}(I : X).$
\end{prop}

\begin{proof}
The proof is very similar to that of above-mentioned proposition; for the sake of completeness, we will include all details. Without loss of generality, we may assume that $X=Y.$ By Theorem \ref{vcb-show}, we have that the function $t\mapsto g(t, y(t)),$ $t\geq 0$ is Stepanov $p$-almost periodic. Since
\begin{align*}
f(t, x(t))=\bigl[g(t, x(t))-g(t, y(t))\bigr]+g(t, y(t))+q(t, x(t)),\quad t\geq 0,
\end{align*}
it suffices to show that
\begin{align}\label{por}
\lim_{t\rightarrow +\infty}\Biggl( \int^{t+1}_{t}\bigl\| g(s, x(s))-g(s, y(s))\bigr\|^{p} \, ds\Biggr)^{1/p}=0
\end{align}
and
\begin{align}\label{pol}
\lim_{t\rightarrow +\infty}\Biggl( \int^{t+1}_{t}\bigl\| q(s, x(s))\bigr\|^{p} \, ds\Biggr)^{1/p}=0.
\end{align}
For the estimate \eqref{por}, we can argue as in the proof of estimate \cite[(2.12)]{comp-adv}: by \eqref{vbnmp} and the H\"older inequality, we have that
\begin{align*}
\Biggl( \int^{t+1}_{t}\bigl\| g(s, x(s))&-g(s, y(s))\bigr\|^{p} \, ds\Biggr)^{1/p}
\\ & \leq \Biggl( \int^{t+1}_{t}L_{g}(s)^{p}\bigl\| x(s)-y(s)\bigr\|^{p} \, ds\Biggr)^{1/p}
\\ & \leq \Biggl( \int^{t+1}_{t}L_{g}(s)^{r}\, ds\Biggr)^{1/r} \Biggl( \int^{t+1}_{t}\bigl\| x(s)-y(s)\bigr\|^{q} \, ds\Biggr)^{1/q}
\\ & =\Biggl( \int^{t+1}_{t}L_{g}(s)^{r}\, ds\Biggr)^{1/r} \Biggl( \int^{t+1}_{t}\bigl\| z(s)\bigr\|^{q} \, ds\Biggr)^{1/q},\quad t\geq 0.
\end{align*}
Hence, \eqref{por} holds on account of $S^{r}$-boundedness of function $L_{g}(\cdot ) $ and assumption $\hat{z}\in C_{0}([0,\infty) : L^{q}([0,1]:X)).$ The proof of \eqref{pol} follows from assumption $\hat{q}\in C_{0}([0,\infty) \times X : L^{p}([0,1]:X))$
and relative compactness of set $ K' =\{x(t) : t \in I \setminus E'\}$
in $X.$
\end{proof}

Similarly, for the class of (equi-)Weyl $p$-almost periodic functions, we have the following result which is not comparabe with \cite[Theorem 3]{mellah} in the case of consideration of equi-Weyl $p$-almost periodic functions, with $I={\mathbb R}$ and $X=Y:$

\begin{thm}\label{vcb-prim} 
Suppose that the following conditions hold:
\begin{itemize}
\item[(i)] $f \in (e-)W_{ap,{\bf K}}^{p}(I \times Y:X)$ with  $p > 1, $ and there exist a number  $ r\geq \max (p, p/p -1)$ and a function $ L_{f}\in L_{S}^{r}(I) $ such that \eqref{vbnmp} holds.
\item[(ii)] $x \in (e-)W_{ap}^{p}(I : Y),$ and there exists a set $E \subseteq I$ with $m (E)= 0$ such that
$ K :=\{x(t) : t \in I \setminus E\}$
is relatively compact in $Y.$ 
\item[(iii)] For every $\epsilon>0,$ there exist two numbers $l>0$ and $L>0$ such that any interval $I'\subseteq I$ of length $L$ contains a number $\tau\in I'$ such that
\begin{align}\label{prc-besik}
\sup_{t\in I,u\in K}\Biggl[ \frac{1}{l}\int_{t}^{t+l}\bigl \| f(s+\tau,u) -f(s,u)\bigr\|^{p}\, ds\Biggr]^{1/p} \leq \epsilon 
\end{align}
and
\begin{align}\label{prcc-besik}
\sup_{t\in I}\Biggl[ \frac{1}{l}\int_{t}^{t+l}\bigl \| x(s+\tau) -x(s)\bigr\|^{p}_{Y}\, ds\Biggr]^{1/p} \leq \epsilon 
\end{align}
in the case of consideration of equi-Weyl $p$-almost periodic functions, resp., 
there exists a finite number $L>0$ such that any interval $I'\subseteq I$ of length $L$ contains a number $\tau\in I'$ satisfying that
there exists a number $l(\epsilon,\tau)>0$ so that \eqref{prc-besik}-\eqref{prcc-besik} hold for all numbers $l\geq l(\epsilon,\tau),$
in the case of consideration of Weyl $p$-almost periodic functions.
\end{itemize}
Then $q:=pr/p+r \in [1, p)$ and $f(\cdot, x(\cdot))\in (e-)W_{ap}^{q}(I : X).$
\end{thm}

\begin{proof}
Without loss of generality, we may assume that $X=Y.$
Since the function $ L_{f}(\cdot)$ is Stepanov $r$-bounded, equivalently, Weyl $r$-bounded,
the measurability and $S^{p}$-boundedness of function $f(\cdot, x(\cdot))$ follow similarly as in the proof of \cite[Theorem 2.2]{comp-adv}. Applying the H\"older inequality and an elementary calculation involving the estimate \eqref{vbnmp} and condition (ii), we get that, for every $t,\ \tau \in I$ and $l>0,$
\begin{align*}
\frac{1}{l}& \int^{t+l}_{t}\bigl\| f(s+\tau,x(s+\tau))-f(s,x(s))\bigr\|^{q}\, ds
\\ \leq &  \frac{1}{l}\Biggl[ \Biggl( \int^{t+l}_{t}L_{f}^{r}(s+\tau)\, ds\Biggr)^{1/r} \Biggl( \int^{t+l}_{t}\bigl\|x(s+\tau)-x(s)\bigr\|^{p}\, dt\Biggr)^{1/p}
\\  & +\Biggl( \int^{t+l}_{t}\bigl\| f(s+\tau,x(s))-f(s,x(s))\bigr\|^{q}\, ds\Biggr)^{1/q}\Biggr]
\\ \leq &  \frac{1}{l}\Biggl[ \Biggl( \int^{t+l}_{t}L_{f}^{r}(s+\tau)\, ds\Biggr)^{1/r} \Biggl( \int^{t+l}_{t}\bigl\|x(s+\tau)-x(s)\bigr\|^{p}\, dt\Biggr)^{1/p}
\\ & +\Biggl( \int^{t+l}_{t}\Bigl(\sup_{u\in K}\bigl\| f(s+\tau,u)-f(s,u)\bigr\|\Bigr)^{q}\, ds\Biggr)^{1/q}\Biggr].
\end{align*}
The remaining part of proof is almost the same for both classes of functions, equi-Weyl $p$-almost periodic functions and  Weyl $p$-almost periodic functions; because of that, we will consider only the first class up to the end of proof. Let $\epsilon>0$ be given.  By (iii), there exist two numbers $l>0$ and $L>0$ such that any interval $I'\subseteq I$ of length $L$ contains a number $\tau\in I'$ such that \eqref{prc-besik}-\eqref{prcc-besik} hold. Since the validity of \eqref{prc-besik}-\eqref{prcc-besik} with given numbers $l>0$ and $\tau\in I$ implies the validity of \eqref{prc-besik}-\eqref{prcc-besik} with numbers $nl$ and $\tau\in I$ ($n\in {\mathbb N}$), we may assume that the number $l>0$ is as large as we want to be. Then, due to Lemma \ref{bes}, we obtain the existence of a finite number $M>0$ such that:
$$ 
\frac{1}{l}\Biggl( \int^{t+l}_{t}L_{f}^{r}(s+\tau)\, ds\Biggr)^{1/r} \leq Ml^{(1/r)-1}\|L_{f}\|_{W^{r}},\quad t\in I
$$
and
\begin{align*}
&\frac{1}{l}\Biggl( \int^{t+l}_{t}L_{f}^{r}(s+\tau)\, ds\Biggr)^{1/r} \Biggl( \int^{t+l}_{t}\bigl\|x(s+\tau)-x(s)\bigr\|^{p}\, dt\Biggr)^{1/p}
\\ \leq & M l^{(1/p)+(1/r)-1}\|L_{f}\|_{W^{r}}=l^{(1/q)-1}\|L_{f}\|_{W^{r}}\leq \|L_{f}\|_{W^{r}},\quad t\in I.
\end{align*}
For the estimation of term
$$
\frac{1}{l}\Biggl( \int^{t+l}_{t}\Bigl(\sup_{u\in K}\bigl\| f(s+\tau,u)-f(s,u)\bigr\|\Bigr)^{q}\, ds\Biggr)^{1/q},\quad t\in I
$$
we can use the trick employed for proving \cite[Lemma 2.1]{comp-adv}.  
Since $K$ is totally bounded, there exist an integer $k\in {\mathbb N}$ and a finite subset $\{x_{1},\cdot \cdot \cdot,x_{k}\}$ of $K$ such that $K\subseteq \bigcup_{i=1}^{k}B(x_{i},\epsilon),$ where $B(x,\epsilon):=\{y\in X : \|x-y\|\leq \epsilon\}.$ Applying Minkowski's inequality and a simple argumentation similar to that used in the proof of above-mentioned lemma, we get the existence of a finite positive real number $c_{q}>0$ such that
\begin{align*}
\frac{1}{l}&\Biggl( \int^{t+l}_{t}\Bigl(\sup_{u\in K}\bigl\| f(s+\tau,u)-f(s,u)\bigr\|\Bigr)^{q}\, ds\Biggr)^{1/q}
\\ \leq & \frac{c_{q}}{l}\Biggl[ \epsilon \Biggl(\int^{t+l}_{t} \bigl[ L_{f}^{q}(s+\tau)+L_{f}^{q}(s)\bigr] \, ds\Biggr)^{1/q}+\sum_{i=1}^{k}\Biggl(\int^{t+l}_{t} \bigl\| f(s+\tau,x_{i})-f(s,x_{i})
\bigr\|^{q}\, ds \Biggr)^{1/q}\Biggr].
\end{align*}  
The term $\frac{1}{l}(\int^{t+l}_{t} [ L_{f}^{q}(s+\tau)+L_{f}^{q}(s)] \, ds)^{1/q}$ can be estimated by using Lemma \ref{bes} in the following way:
\begin{align*}
\frac{1}{l}\Biggl(\int^{t+l}_{t} &\bigl[ L_{f}^{q}(s+\tau)+L_{f}^{q}(s)\bigr] \, ds\Biggr)^{1/q}\leq \frac{1}{l} \Biggl(\int^{t+l+\tau}_{t+\tau} L_{f}^{q}(s)\, ds\Biggr)^{1/q}
+\frac{1}{l}\Biggl(\int^{t+l}_{t} L_{f}^{q}(s)\, ds\Biggr)^{1/q}
\\ \leq & Ml^{(-1/r)+(1/q)-1}\bigl\|L_{f}\bigr\|_{W^{r}}l^{1/r}\leq M\bigl\|L_{f}\bigr\|_{W^{r}},\quad t\in I.
\end{align*}
Similarly, using Lemma \ref{bes} and (iii), we get
\begin{align*}
\frac{1}{l}&\sum_{i=1}^{k}\Biggl(\int^{t+l}_{t} \bigl\| f(s+\tau,x_{i})-f(s,x_{i})
\bigr\|^{q}\, ds \Biggr)^{1/q}
\\ \leq & \frac{1}{l}l^{(1/q)-(1/p)}\sum_{i=1}^{k}\Biggl(\int^{t+l}_{t} \bigl\| f(s+\tau,x_{i})-f(s,x_{i})
\bigr\|^{p}\, ds \Biggr)^{1/p}\leq \epsilon l^{(1/q)-1},\quad t\in I.
\end{align*}
This completes the proof of theorem.
\end{proof}

\begin{rem}\label{obe-klase}
To the best knowledge of the author, it is not known whether the assumptions $f \in (e-)W_{ap}^{p}(I \times Y:X)$ and $x \in (e-)W_{ap}^{p}(I : Y)$
imply the validity of condition (iii), as for the class of Stepanov $p$-almost periodic functions. 
\end{rem}

The following result for the class of Weyl $p$-almost periodic functions can be also deduced with the help of argumentation contained in \cite{comp-adv} (compare with Theorem \ref{vcb-show}, where we have analyzed the Stepanov class):

\begin{thm}\label{vcb-show-weyll}
Suppose that $p,\ q\in [1,\infty) ,$ $r\in [1,\infty],$ $1/p=1/q+1/r$ and the following conditions hold:
\begin{itemize}
\item[(i)] $f \in W_{ap,{\bf K}}^{p}AP(I \times Y : X)  $ and there exists a function $ L_{f}\in L_{S}^{r}(I) $ such that \eqref{vbnmp} holds.
\item[(ii)] $x \in W_{ap}^{q}AP (I: Y),$ and there exists a set ${\mathrm E} \subseteq I$ with $m ({\mathrm E})= 0$ such that
$ K :=\{x(t) : t \in I \setminus {\mathrm E}\}$
is relatively compact in $Y.$ 
\item[(iii)] For every $\epsilon>0,$ there exists a finite number $L>0$ such that any interval $I'\subseteq I$ of length $L$ contains a number $\tau\in I'$ satisfying that
there exists a number $l(\epsilon,\tau)>0$ so that
\eqref{prc-besik} holds for all numbers $l\geq l(\epsilon,\tau)$ and \eqref{prcc-besik} holds for all numbers $l\geq l(\epsilon,\tau),$ with the number $p$ replaced by $q$ therein.
\end{itemize}
Then $f(\cdot, x(\cdot)) \in W_{ap}^{p}AP(I : X).$
\end{thm}

After proving Theorem \ref{vcb-prim}, the subsequent composition principle for asymptotically (equi-)Weyl $p$-almost periodic functions follows almost immediately; cf. also \cite[Theorem 3.4]{irkutsk-prim} for a similar result in this direction.

\begin{thm}\label{rora} 
Suppose that $p > 1, $ $ r\geq \max (p, p/p -1),$ $q=pr/p+r,$ and the conditions \emph{(i)-(iii)} of \emph{Theorem \ref{vcb-prim}} hold with the interval $I=[0,\infty)$ and the functions $f(\cdot,\cdot),$  $x(\cdot)$ replaced therein with the functions $g(\cdot,\cdot),$ $y(\cdot).$ Suppose, further, that the following holds:
\begin{itemize}
\item[(i)] The function $Q:=f-g : [0,\infty ) \times Y \rightarrow X$ is in class $(e-)W^{q'}_{0,{\bf K}}([0,\infty )\times Y:X) $ for some number $q'\in [1,\infty).$
\item[(ii)] The function $z : [0,\infty ) \rightarrow Y$ is in class $(e-)W^{q''}_{0}([0,\infty ):Y) $ for some number $q''\in [1,\infty).$
\item[(iii)] $x(t)=y(t)+z(t)$ for a.e. $t\geq 0,$ and there exists a set ${\mathrm E} \subseteq I$ with $m ({\mathrm E})= 0$ such that
$ K :=\{x(t) : t \in I \setminus {\mathrm E}\}$
is relatively compact in $Y.$
\end{itemize}
Then the mapping $t\mapsto f(t,x(t)),$ $t\geq 0$ is in class $$(e-)W^{q}_{ap}([0,\infty): X) +(e-)W^{q'}_{0}([0,\infty ):X)+(e-)W^{q'''}_{0}([0,\infty ):X) ,$$
provided $q'''\in [1,\infty)$ and $1/r+1/q''=1/q'''.$
\end{thm}

\begin{proof}
It is clear that 
\begin{align*}
f(t, x(t))=\bigl[g(t, x(t))-g(t, y(t))\bigr]+g(t, y(t))+Q(t, x(t)),\quad t\geq 0.
\end{align*}
By Theorem \ref{vcb-prim}, we know that $g(\cdot, y(\cdot))\in (e-)W_{ap}^{q}([0,\infty) : X).$ Keeping in mind (i) and (iii), 
the function $t\mapsto Q(t,x(t)),$ $t\geq 0$ is in class $(e-)W^{q'}_{0}([0,\infty ):X)$ by definition (see the notions of classes $W^{p}_{0,{\bf K}}(I \times Y:X)$ and $e-W^{p}_{0,{\bf K}}(I \times Y:X)$ introduced in Definition \ref{stea}).
Therefore,
it suffices to show that the mapping $t\mapsto g(t, x(t))-g(t, y(t)),$ $t\geq 0$ is in class $(e-)W^{q'''}_{0}([0,\infty ):X) .$ But, this follows similarly as in the proof of  \cite[Theorem 3.4]{irkutsk-prim}, with the exponents $p,\ q,\ r$ replaced therein with the exponents $q''',\ q'',\ r,$ respectively. 
\end{proof}

An analogue of \cite[Theorem 3.4]{irkutsk-prim} for the class of asymptotically Weyl $p$-almost periodic functions can be also deduced by means of Theorem \ref{vcb-show-weyll}. 

Before we move ourselves to the next subsection, we 
would like to mention that S. Abbas has recently proved a composition principle for the class of Weyl
pseudo almost automorphic functions in \cite{sead-abas}.

\subsection[Applications to semilinear Cauchy inclusions]{Applications to semilinear Cauchy inclusions}\label{semilinear}

The composition principles for (asymptotically) Stepanov $p$-almost periodic functions can be almost straightforwardly incorporated in the study of existence and uniqueness of (asymptotically) almost periodic solutions for a great number of different types of abstract semilinear Cauchy inclusions in Banach spaces with Stepanov forcing coefficients. For example, Theorem 
\ref{vcb-show} is applicable in the study of qualitative properties of solutions of the abstract semilinear Volterra integral inclusion 
\begin{align*}
u(t)\in {\mathcal A}\int^{t}_{-\infty}a(t-s)u(s)\, ds +f(t,u(t)),\ t\in {\mathbb R},
\end{align*}
where $a \in L_{loc}^{1}([0,\infty)),$ $a\neq 0,$ $f : {\mathbb R} \rightarrow X$ is Stepanov $p$-almost periodic, ${\mathcal A}$ is a closed multivalued linear operator on $X$
and certain extra assumptions are satisfied (see e.g. the articles \cite{cuevas-lizama-duo} by C. Cuevas, C. Lizama and \cite{hernan-lizama} by H. R. Henr\' iquez, C. Lizama as well as \cite{nova-mono} for the notion and basic results about multivalued linear operators). We can also analyze almost periodic solutions of the following abstract Cauchy fractional relaxation inclusion
\begin{align*}
D_{t,+}^{\gamma}u(t)\in -{\mathcal A}u(t)+f(t,u(t)),\ t\in {\mathbb R},
\end{align*}
where $D_{t,+}^{\gamma}$ denotes the Weyl Liouville fractional derivative of order $\gamma \in (0,1),$ 
$f : {\mathbb R} \rightarrow X$ is Stepanov $p$-almost periodic and ${\mathcal A}$ is a closed multivalued linear operator on $X.$ For example, the assertions of \cite[Theorem 4.3, Theorem 4.5]{jms} can be reformulated in our framework.

Combining Theorem 
\ref{vcb-show} and Proposition \ref{biblonja}, we can examine the existence and uniqueness of asymptotically almost periodic solutions of fractional relaxation inclusions with
Caputo derivatives. More precisely, 
let $\gamma \in (0,1)$ and let ${\mathcal A}$ be a closed multivalued linear operator on $X.$ Of concern is the abstract semilinear Cauchy inclusion
\[
\hbox{(DFP)}_{f,\gamma,s} : \left\{
\begin{array}{l}
{\mathbf D}_{t}^{\gamma}u(t)\in {\mathcal A}u(t)+f(t,u(t)),\ t> 0,\\
\quad u(0)=x_{0},
\end{array}
\right.
\]
where ${\mathbf D}_{t}^{\gamma}$ denotes the Caputo fractional derivative of order $\gamma,$ $x_{0}\in X$ and 
$f : [0,\infty) \times X \rightarrow X$ is Stepanov $p$-almost periodic. Using the above-mentioned results, we can simply reformulate the assertions of \cite[Theorem 2.3, Corollary 2.1, Theorem 2.6, Corollary 2.3]{publ} in our framework.

As in a great number of our recent research studies, the results obtained in the first part of this section 
can be successfully applied in the analysis of degenerate Poisson heat equation and its fractional analogues (cf. \cite{nova-mono} for further information in this direction). Let $\Omega$ be a bounded domain in ${\mathbb R}^{n},$ $b>0,$ $m(x)\geq 0$ a.e. $x\in \Omega$, $m\in L^{\infty}(\Omega),$ $1<p<\infty$ and $X:=L^{p}(\Omega).$
We can apply Theorem 
\ref{vcb-show} for proving certain results about the existence and uniqueness of almost periodic solutions of semilinear Poisson heat equation
\[
D^{\gamma}_{t,+}[m(x)v(t,x)]=(\Delta -b )v(t,x) +f(t,m(x)v(t,x)),\quad t\in {\mathbb R},\ x\in {\Omega}
\]
with Weyl Liouville fractional derivatives. Taken together, Theorem 
\ref{vcb-show} and Proposition \ref{biblonja} are applicable in the study of existence and uniqueness of asymptotically almost periodic solutions of the fractional semilinear Poisson heat equation
\[\left\{
\begin{array}{l}
{\mathbf D}_{t}^{\gamma}[m(x)v(t,x)]=(\Delta -b )v(t,x) +f(t,m(x)v(t,x)),\quad t\geq 0,\ x\in {\Omega},\\
v(t,x)=0,\quad (t,x)\in [0,\infty) \times \partial \Omega ,\\
 m(x)v(0,x)=u_{0}(x),\quad x\in {\Omega}
\end{array}
\right.
\]
with Caputo fractional derivatives.

The spaces of (equi-)Weyl $p$-almost periodic functions are not complete with the respect to the Weyl norm, as it is well known. Without any doubt, this seriously hinders possibilities to apply 
Theorem \ref{vcb-prim}, Theorem \ref{vcb-show-weyll} and Theorem \ref{rora} in the analysis of existence and uniqueness of (asymptotically) almost Weyl $p$-almost periodic solutions 
of abstract semilinear Cauchy inclusions. The existence and uniqueness of bounded, continuous, equi-Weyl $p$-almost periodic solutions of the abstract semilinear
differential equations of first order has been initiated in \cite[Section 3.3]{mellah} (cf. also \cite[Theorem 3.5, Theorem 3.6]{irkutsk-prim}). Because of its non-trivial character, we will continue this theme somewhere else.

The paper is dedicated to the memory of old Professor Bogoljub Stankovi\' c. For his contributions in the field of generalized almost periodic functions, we would like to recommend for the reader the articles \cite{stankovic}-\cite{stankovic1}.

\bibliographystyle{amsplain}

\end{document}